\documentclass[titlepage,12pt]{article}

\usepackage{amsmath, amsthm, amssymb,amscd, bbm}
\usepackage{hyperref}
\hypersetup{
    colorlinks   = true,
    citecolor    = blue,
    urlcolor=blue,
}
\usepackage[titletoc,title]{appendix}
\usepackage[shortlabels]{enumitem}

\newcommand{\R}{{\mathbb R}}

\newcommand{\N}{{\mathbb N}}

\newcommand{\abs}[1]{\left| #1 \right|}

\newtheorem{theorem}{Theorem}[section]
\newtheorem{lemma}[theorem]{Lemma}
\newtheorem{proposition}[theorem]{Proposition}
\newtheorem{corollary}[theorem]{Corollary}

\theoremstyle{remark}

\newtheorem{assumption}[theorem]{Assumption}

\theoremstyle{remark}

\newcommand{\norm}[1]{\left|\left|#1\right|\right|}

\newcommand{\hypo}{\mathrm{hypo}\,}
\newcommand{\conv}{{\mathrm{conv}\,}}

\newcommand{\dom}{{\mathrm{dom}\,}}

\newcommand{\argmax}{\mathop{\rm argmax}}
\newcommand{\argmin}{\mathop{\rm argmin}}
\newcommand{\setint}{\operatorname{int}} 
\newcommand{\lev}{\operatorname{lev}} 

\newcommand{\Ex}{{I\kern-.35em E}}
\newcommand{\Pro}{{I\kern-.35em P}}

\newcommand{\reals}{{I\kern-.35em R}}
\newcommand{\Reals}{\overline{I\kern-.35em R}}

\newcommand{\nats}{{I\kern-.35em N}}

\usepackage{lastpage}
\usepackage{graphicx}
\usepackage{comment}

\begin{document}
\baselineskip=15pt

\begin{titlepage}
\vglue 0.5cm
\begin{center}
\begin{large}
{\bf Maximum a Posteriori Estimators as a Limit of Bayes Estimators


}

\smallskip
\end{large}
\vglue 3.0truecm
\begin{tabular}{lcl}
  \begin{large} {\sl  Robert Bassett
                                  } \end{large} & \ \ {\phantom{\&}} \ \ &
  \begin{large} {\sl Julio Deride 
                                  } \end{large} \\
  \\
  Mathematics  && Mathematics \\
  Univ. California, Davis && Univ. California, Davis \\
  rbassett@math.ucdavis.edu && jderide@math.ucdavis.edu  
\end{tabular}
\end{center}
\vskip 1.5truecm
\noindent {\bf Abstract}. \quad Maximum a posteriori and Bayes
estimators are two common methods of point estimation in Bayesian
Statistics. It is commonly accepted that maximum a
posteriori estimators are a limiting case of Bayes estimators with
$0$-$1$ loss. In this
paper, we provide a counterexample which shows that in general this
claim is false. We then correct the claim that by providing a
level-set condition for posterior densities such that the result holds.
Since both estimators are defined in terms of optimization problems,
the tools of variational analysis find a natural application to
Bayesian point estimation.

\vskip 1.5truecm
\halign{&\vtop{\parindent=0pt
   \hangindent2.5em\strut#\strut}\cr
\hglue 1.50cm  point estimation.  \cr\cr
{\bf Date}:\quad \ \today \cr}
\end{titlepage}
\baselineskip=15pt

\section{Introduction}

The purpose of this article is to relate two point estimates in 
Bayesian estimation: the maximum a posteriori (MAP) estimator and Bayes
estimator. Both the MAP and Bayes estimator are defined in terms of optimization
problems, so that any connection between MAP and Bayes estimators can be 
extended to a connection between corresponding optimization problems. 
It is commonly accepted (\cite[\S4.1.2]{Bayes1} 
\cite[Thm.2.4.3]{Ge05contBE},\cite[\S7.6.5]{Lee12bayesian}) that 
\emph{MAP estimation is the limit of Bayes estimation}. This
relationship is appealing from a theoretical perspective because it
allows MAP estimators to be subsumed by the statistical analysis and intuition of Bayes
estimators. However, this assertion must be carefully studied, because
it is not true in the general setting proposed in much of the
literature. We apply the theory of variational analysis, a common tool 
used in the approximation of optimization problems, to
investigate the relationship between MAP and Bayes estimators.

This paper revises the relationship between MAP and Bayes estimators,
placing the theory on a solid mathematical foundation. First, we
provide a counterexample to the commonly accepted notion of MAP
estimators as a limit of Bayes estimators having $0$-$1$ loss.
We then provide additional conditions and resurrect the limiting relationship
between the estimators by relying on the theory of variational
analysis. Because each of the estimators is defined as the
maximizer of a certain optimization problem, a natural setting for
the analysis of this problem is the space of upper semi-continuous
functions, where we use the appropriate topology for convergence of 
maximizers, namely \emph{hypo-convergence}. In general, the approach
in this paper is applicable to any estimator defined in terms of an optimization
problem.

This paper is not the first to apply variational
analytic techniques to statistical estimation. One of the earliest
examples is the notion of epi-convergence in distribution, which 
was introduced in \cite{SalWets86} and developed further in \cite{Pflug91} and 
\cite{Knight99}. The majority of
applications of epi-convergence in distribution
\cite{Pflug95,Knight01,Knight00}, and general applications of
variational analytic theory to statistical estimation problems
\cite{KingWets91,AtWe94epigraphical,Shapiro91,DupaWets88,KingRock93,Geyer94}, 
have focused on asymptotic consistency of estimators: the notion of
approximating a \emph{true} optimization problem by one constructed
from a finite sample, as the size of the sample tends to infinity. 

This paper differs from the literature in a number of ways. First, we
consider the limit of optimization problems defined over the same
measure space. Instead of contrasting the empirical solution to the
true, we instead assume that the function to be optimized is changing,
but that the ``true'' underlying measure is known a priori. In this
sense, we focus more on the approximation of loss functions than on
the measure over which their expected value is taken. We also focus on
almost-sure convergence results, as opposed to the weaker notion of
distributional convergence. Lastly, the convergence in this paper
deals explicitly with the Bayesian framework and the relationship
between MAP and Bayes estimators.

The rest of this paper is organized as follows. In section
\ref{sec:2}, a mathematical formulation of the Bayesian point estimation
framework is reviewed, including the commonly-accepted argument of the relationship
between MAP and Bayes estimators. Section \ref{sec:3} gives 
variational analytic background, introducing the notion of an upper 
semi-continuous density and the main convergence results for the hypo-convergence
topology. Section \ref{sec:4} provides an example where a sequence
of Bayes estimators corresponding to $0$-$1$ loss does not converge to a MAP
estimator. In light of the counterexample, some condition is required for MAP estimation to be a limiting case of Bayesian estimation. The informal arguments in the literature give the misleading impression that no condition is needed.
We provide a necessary condition in section \ref{sec:5}, and use it to prove positive results
relating MAP and Bayes Estimators.

We conclude this section with a comment on our notation. Preserving the notation in \cite{VaAn}, 
we use Greek letters $\nu$ and $\eta$ to denote sequence indices. We use superscript indexing, so that $x^{\nu}$ is an index of the sequence $x$.
\section{Bayesian Background}\label{sec:2}

In this section we review necessary Bayesian preliminaries. In point
estimation, we wish to infer the value of an unknown parameter
$\theta$ from an observation $x$.

A \emph{parametric model} is a family of distributions 
$p(\cdot | \theta): \mathcal{X} \to \mathbb{R}$ indexed by a
set $\Theta \subseteq \mathbb{R}^{n}$. The set $\Theta$ is referred to
as the parameter space, and the measurable space $\mathcal{X}$ is the
called the sample space. We assume that each of these distributions
is absolutely continuous with respect to Lebesgue measure; hence each 
$p(\cdot | \theta)$ is a probability density function.

A \emph{Bayesian model} is a parametric model and a distribution
$\pi$ on $\Theta$. The distribution $\pi$ is called a $\emph{prior}$
on the parameter space $\Theta$. The Bayesian point estimation problem
is the following: Given an observation $x$ of a random variable $X$ on $\mathcal{X}$,
find $\hat{\theta} \in \Theta$ such that $p(\cdot | \hat{\theta})$ is a ``best
choice'' for the distribution of $X$ among distributions in the parametric 
model. In this sense, a Bayesian point estimate is a function
$\hat{\theta}: \mathcal{X} \to \Theta$ which takes observations to
parameter values. Because we only refer to Bayesian point estimates in
this paper, we simply refer to them as estimators.

Given a Bayesian model and an observation $x$, we define a
posterior distribution on the parameter space $\Theta$ through Bayes'
rule
\begin{equation} \label{posterior}
\pi(\theta|x) = \frac{p(x|\theta) \pi(\theta)}{\int_{z \in \Theta}
p(x|z) \pi(z) \, dz}
\end{equation}
We assume that for each $x\in \mathcal{X}$, $\int_{z \in \Theta}
p(x|z) \pi(z) \, dz$ is finite and nonzero, so that
\eqref{posterior} defines a density. By taking
$\pi(\theta|x) = 0$ outside of $\Theta$, we extend
the posterior density so that it is defined on all of $\R^n$ for
each $x$. Hence, without loss of generality, we assume that $\Theta =
\R^n$.

A common method of point estimation is through Bayes
estimators. Given a loss function $L: \Theta \times \Theta \to [0,
\infty)$ which quantifies cost for discrepancies in the
true and estimated parameter value, a
\emph{Bayes estimator} is an estimator $\hat{\theta}_{B}$
which minimizes posterior expected loss.

\begin{equation}\label{Bayes}
\hat{\theta}_{B}(x) \in \argmin_{\theta \in \Theta} \mathbb{E}^{z|x}\left[
L(\theta, z) \right] \ = \argmin_{\theta \in \Theta} 
\int_{z \in \Theta} L(\theta, z) \pi(z|x) \, dz.
\end{equation}

The flexibility and ubiquity of \eqref{Bayes} should not be understated.
With different choices of loss functions, one can define the posterior
mean, median and other quantiles \cite[2.5]{Bayes1}, as well as a variety of robust
variants through expected loss minimization \cite[10.6]{ElemStats}. For our purposes, we will
focus on one particular family of loss functions, the $0$-$1$ loss
functions. The $0$-$1$ loss function $L^c$ is defined for any $c>0$ as

\begin{equation}
L^{c}(\theta, z) = \begin{cases} 0 & \text{ for } \norm{\theta -z} <
\frac{1}{c} \\ 1 & \text{ otherwise.} \end{cases}
\end{equation}
In the above and what follows, $\norm{\cdot}$ denotes the standard Euclidean norm. 
Because of the topological nature of the arguments of that follow, any equivalent norm
could be used instead. We focus on standard Euclidean for ease of exposition.

The rest of this paper will deal almost exclusively with the $0$-$1$ loss,
so we emphasize our notation: superscript $L$ denotes the $0$-$1$ loss 
function. We also denote by $\hat{\theta}_{B}^{c}$ the Bayes estimator
associated with the $0$-$1$ loss function $L^c$.

Another popular estimation procedure maximizes the posterior density
directly. This defines a \emph{maximum a posteriori estimator},
$\hat{\theta}_{MAP}$, which is given by the set of modes of the posterior
distribution
$$\hat{\theta}_{MAP}(x) \in \argmax_{\theta} \pi(\theta|x).$$

This estimator can be interpreted as an analog of maximum likelihood
for Bayesian estimation, where the distribution has become a posterior.
A number of sources (\cite[\S4.1.2]{Bayes1}
\cite[Thm.2.4.3]{Ge05contBE},\cite[\S7.6.5]{Lee12bayesian}, \cite{Paper1},
\cite{LectNotes}) claim that maximum a posteriori estimators are
limits of Bayes estimators, in the following sense. 
Consider the sequence of 0-1 loss functions,
$\{L^\nu:\reals^n\times\reals^n\to[0,+\infty)\}_{\nu \in \nats}$,
defined as 
\begin{equation} \label{01Seq}
L^\nu(\theta,z)=\begin{cases} 0 & {\rm
for}\,\|\theta-z\|<\frac{1}{\nu}\\
1 & {\rm otherwise} \end{cases},
\end{equation}
and define $\hat{\theta}_{B}^\nu$ as the Bayes estimator associated to the
loss function $L^\nu$, for each $\nu$,
i.e.,
\[ \hat{\theta}_{B}^{\nu}(x) \in \argmin_{\theta\in\Theta}
\Ex^{z|x} \left[ L^\nu(\theta,z) \right] .\]
First we translate the Bayes estimator to a maximization problem.
\begin{align} \label{min2max}
\hat{\theta}_{B}^{\nu}(x)& \in \argmin_{\theta\in\Theta} \int_{z \in
\Theta}
L^\nu(\theta,z)\pi(z|x)\,dz \nonumber\\
&=\argmin_{\theta\in\Theta}\left(1-\int_{\|\theta-z\|<\frac{1}{\nu}}
\pi(z|x)\,dz\right) \nonumber\\
&=\argmax_{\theta\in\Theta} \int_{\|\theta-z\|<\frac{1}{\nu}}
\pi(z|x)\,dz.
\end{align}
The claim is that the sequence $\hat{\theta}^{\nu}_{B}$ converges to
$\hat{\theta}_{MAP}$. When the justification is provided, it proceeds as follows.
Taking the limit as $\nu\to\infty$, we have
\begin{align} \label{limitEst}
\lim_{\nu\to\infty} \hat{\theta}_{B}^{\nu}(x)& \in \lim_{\nu\to\infty}
\argmax_{\theta\in\Theta} \int_{\|\theta-z\|<\frac{1}{\nu}}
\pi(z|x)\,dz  \\
&= \argmax_{\theta\in\Theta} \pi(\theta|x) \nonumber\\
&=\hat{\theta}_{MAP}(x). \nonumber
\end{align}

This justification does not hold in general, and in fact MAP estimators are not
necessarily a limiting case of Bayes estimators under the $0$-$1$ loss indicated
above. In section \ref{sec:4}, we exhibit a continuous and unimodal
posterior density which is a counterexample to the claim. The problem 
lies in the limiting argument in \eqref{limitEst}--the limit of
maximizers is not a maximizer without additional conditions, which we
establish in section \ref{sec:5}.

It is worthwhile to note what is correct about the argument in
\eqref{limitEst}. 
Denoting by $s_n$ the volume of the unit ball in $\R^n$, we have that 
$$\lim_{\nu \to \infty} s_{n} \cdot \nu^{n} \cdot 
\int_{\norm{\theta-z}<\frac{1}{\nu}} \pi(z|x) \, dz =
\pi(\theta|x)$$
$\Theta$-almost everywhere by the Lebesgue differentiation theorem.
This gives that a scaled version of $\int_{\norm{\theta-z} <
\frac{1}{\nu}} \pi(z|x) \, dz$ converges pointwise a.e. to
$\pi(\theta|x)$. But pointwise convergence does not guarantee convergence of
maximizers. This will require the notion of hypo-convergence of
functions, which we introduce in the next section.

\section{Convergence of maximizers for Non-Random Functions}\label{sec:3}
\subsection{General setting}
This section summarizes the main results for convergence of optimization 
problems. The more common approach in the optimization literature is that of minimization, 
rather than maximization, where the theory of epi-convergence is developed for extended
real-valued functions. Here, an adaptation to the maximization setting is presented,
which is more natural in our estimation setting.

A function $f:\reals^n \to \Reals$ is said to be \emph{proper} if $f$ is 
not constantly $-\infty$ and never takes the value $\infty$. The 
\emph{effective domain} of the function $f$ is the set
\[\dom f = \{ x|\, f(x)<\infty\},\]
and its \emph{hypograph} is the set in $\reals^{n+1}$
\[\hypo f = \{ (x, \lambda) |\, \lambda \leq f(x) \}.\]

The function $f$ is called \emph{upper semi-continuous} (usc) if its hypograph is
a closed subset of $\reals^{n+1}$. An equivalent condition is that for
each $\alpha \in \mathbb{R}$ the upper level set of $f$
\[\lev_{\geq \alpha} f = \{ x \in \reals^{n}|\,f(x) \geq \alpha\}\]
is closed. A sequential definition of upper semi-continuity can
also be stated: for each  $x \in\reals^{n}$, and each sequence $x^{\nu}$ converging to $x$,
\[\limsup_{x^{\nu} \to x} f(x^{\nu})\leq f(x).\]

We say that a sequence of functions $f^{\nu}$ 
\emph{hypo-converges} to a function $f$ if both
of the following hold for each $x \in \mathcal{X}$.
\begin{align*}
& \liminf_\nu f^{\nu}(x^{\nu}) \geq f(x)  \text{  for some }
x^{\nu} \to x \\
& \limsup_\nu f^{\nu}(x^{\nu}) \leq f(x)  \text{  for every }
x^{\nu} \to x,
\end{align*}

The notion of hypo-convergence is well-developed because of its
importance in proving properties about sequences of optimization problems.
An equivalent definition of hypo-convergence follows by identifying
each function with its hypograph, and applying
the notion of set convergence \`a-la Painlev\'e-Kuratowski.
Or, one can characterize hypo-convergence via the
Attouch-Wets topology on hypographs. We refer the reader to
\cite[Ch.7]{VaAn} for details on the former and  \cite[Ch.3]{Beer93top} on the latter.

The following alternative characterization of hypo-convergence will be
useful in the sections that follow.
\begin{proposition}{\emph{\cite[7.29]{VaAn}}} \label{HitMiss}
$f^{\nu} \xrightarrow{hypo} f$ with $f$ usc if and only if the
following two conditions hold:
\begin{align*}
\limsup_{\nu} \left( \sup_{B} f^{\nu} \right) \leq \sup_{B} f \text{ for
every compact set } B \subseteq \mathbb{R}^{n} \\
\liminf_{\nu} \left( \sup_{O} f^{\nu} \right) \geq \sup_{O} f \text{ for
every open set } O \subseteq \mathbb{R}^{n}
\end{align*}
\end{proposition}

A sequence of functions $f^{\nu}$ is said to be \emph{eventually
level-bounded} if for each $\alpha \in \mathbb{R}$ the sequence of
upper level sets $\lev_{\geq \alpha} f^{\nu}$ is eventually bounded.
That is, there exists a bounded set $M$, $\alpha \in \mathbb{R}$ and $m \in \mathbb{N}$ 
such that for all $\nu \geq m$, $\lev_{\geq \alpha} f^{\nu} \subseteq M$.

The wide-spread adoption of hypo-convergence in optimization is largely
due to the following theorem and its related extensions
\begin{theorem}{\emph{\cite[7.33]{VaAn}}}
Assume $f^{\nu} \xrightarrow{hypo} f$, where $f^{\nu}$ is eventually
level-bounded and $\{f^{\nu},f\}$ are usc and proper. Then
\[\sup\, f^{\nu} \to \sup\, f\]
and any point which is a limit of maximizers of $f^{\nu}$ maximizes $f$.
\end{theorem}

This theorem is attractive because it effectively gives the
convergence of maximizers for approximations to optimization problems.
In our application to Bayesian point estimation, we wish to establish
a similar theorem providing convergence of estimators for approximations
to MAP and Bayes estimation. We do so in section \ref{sec:5}.

\subsection{Upper Semi-Continuous Densities}
Before proceeding, we list the primary assumption that we use throughout
the paper.
\begin{assumption} \label{assump}
For each $x \in \mathcal{X}$, the random vector $\theta$ has a density
$\pi(\theta|x): \mathbb{R}^{n} \to \overline{\mathbb{R}}$
(with respect to Lebesgue measure $\mu$) which is
continuous almost everywhere. In
other words, for each $x$ there is a set $C \subseteq \mathbb{R}^n$ such 
that $\pi(\theta|x)$ is
continuous at each point in $C$ and $\mu \left(\mathbb{R}^n \setminus C 
\right) = 0$. \end{assumption}

In this framework, we allow a density $\pi(\theta|x)$ to take $\infty$
as a value, so as not to rule out common densities like the gamma
distribution. Obviously a density cannot take $\infty$ as a value on a set of
positive measure.

When dealing with arbitrary measurable functions, the notion of
pointwise value is inherently ambiguous because measurable functions
are only unique up to alteration on sets of measure zero. The added
structure of continuity allows us to define the notion of pointwise
value, by referring to the pointwise value of the unique continuous
representative of the measurable function. We would like to generalize
this to a broader class of functions.

Assumption \ref{assump} provides some minimalist structure for which we can also 
define an unambiguous notion of pointwise value of a function. This is necessary because 
without some method of defining a pointwise value of a density, the notion of maximizing 
the density (as required in MAP estimation) is meaningless.

For each $x$, take $\pi(\theta|x)$ to be its
continuous version on $C$. On $\mathbb{R}^n \setminus C$, take $\pi(\theta|x)$ to be the
upper semi-continuous envelope of $\pi(\theta|x)|_{C}$. That is, for any $\theta \in
\mathbb{R}^n$
\begin{equation} \label{uscDef}
\pi(\theta|x) = \sup \{\limsup \pi(\theta^{\nu}|x) : \theta^{\nu} \to \theta, \theta^{\nu} \subset C
\}.
\end{equation}
The set on the right side of \eqref{uscDef} is nonempty because the
complement of a measure zero set is dense. Furthermore, because
$\mathbb{R}^{n} \setminus C$ has measure zero, the integral of $\pi(\theta|x)$
over $\mathbb{R}^n$ is unchanged by taking the upper semi-continuous
envelope.
We refer to densities which satisfy assumption \ref{assump} and
\eqref{uscDef} as \emph{upper semi-continuous (usc) densities}. In the
remaining sections, densities are assumed to be upper
semi-continuous.

Upper semi-continuous densities are a natural class of functions to
consider because they provide a large amount of flexibility while
still retaining
the notion of pointwise value. They contain continuous densities as a
special case, but also include intuitive concepts like histograms.
Many nonparametric density estimation procedures also produce
functions in this family, e.g. \cite{episplines},
\cite{chan14}, and kernel density estimates with
piecewise continuous kernels.

\section{Counterexample}\label{sec:4}

This section provides a counterexample to the following claim: Any limit
of Bayes estimators $\hat{\theta}_{B}^{c^{\nu}}$ having $0$-$1$ loss
$L^{c^{\nu}}$ with $c^{\nu} \to \infty$ is a MAP estimator. First, we
provide the Bayesian model. Let $\Theta = \R$ and $\mathcal{X} = \R$.
Take $p(x|\theta)$ to be a standard normal distribution. In other words, the parameter $\theta$
has no influence on the distribution of $X$. While trivial
in the fact that $\theta$ does not influence $X$, this example
will facilitate computation and
illustrate the issue with the argument in \eqref{limitEst}. 
Consider a prior density $\pi: \reals \to \reals$ given by
 
\[\pi(\theta)=
\begin{cases}
1-\sqrt{\abs{2\theta}} & \theta \in \left(\frac{-1}{2}, \frac{1}{2}\right) \\
(2^n-1) 4^n (\theta + \frac{1}{8^n} -n) & \theta \in \left[n-\frac{1}{8^n},
n\right], \; n \in \mathbb{N}_{+}  \\
1-\frac{1}{2^n} & \theta \in [n, n+\frac{1}{2^n}-\frac{1}{8^{n}}], \;
n \in \mathbb{N}_{+}\\
(1-2^n)4^n \left(\theta-(n+\frac{1}{2^n}) \right) &
\theta \in \left[n+\frac{1}{2^n}-\frac{1}{8^{n}}, n+\frac{1}{2^n}\right] \; n \in
\mathbb{N}_{+} \\
0 &\text{ otherwise}
\end{cases}.\]
This density is depicted in Figure \ref{fig:2}. One can easily check
that $\pi$ is indeed a continuous density function. Because $\theta$
and $X$ are independent in this Bayesian model, equation
\eqref{posterior} gives that the posterior is equal to the prior. Thus, for the
remainder of this section, we will refer to $\pi$ as the posterior
distribution, and drop the dependence on $x$ in the notation. We also
comment that any choice of parametric model where altering $\theta$ does not
change $p(x|\theta)$ will give rise to the same posterior
distribution--the standard normal assumption is just for concreteness.

\begin{figure}
\centering
  \includegraphics[width=0.75\linewidth]{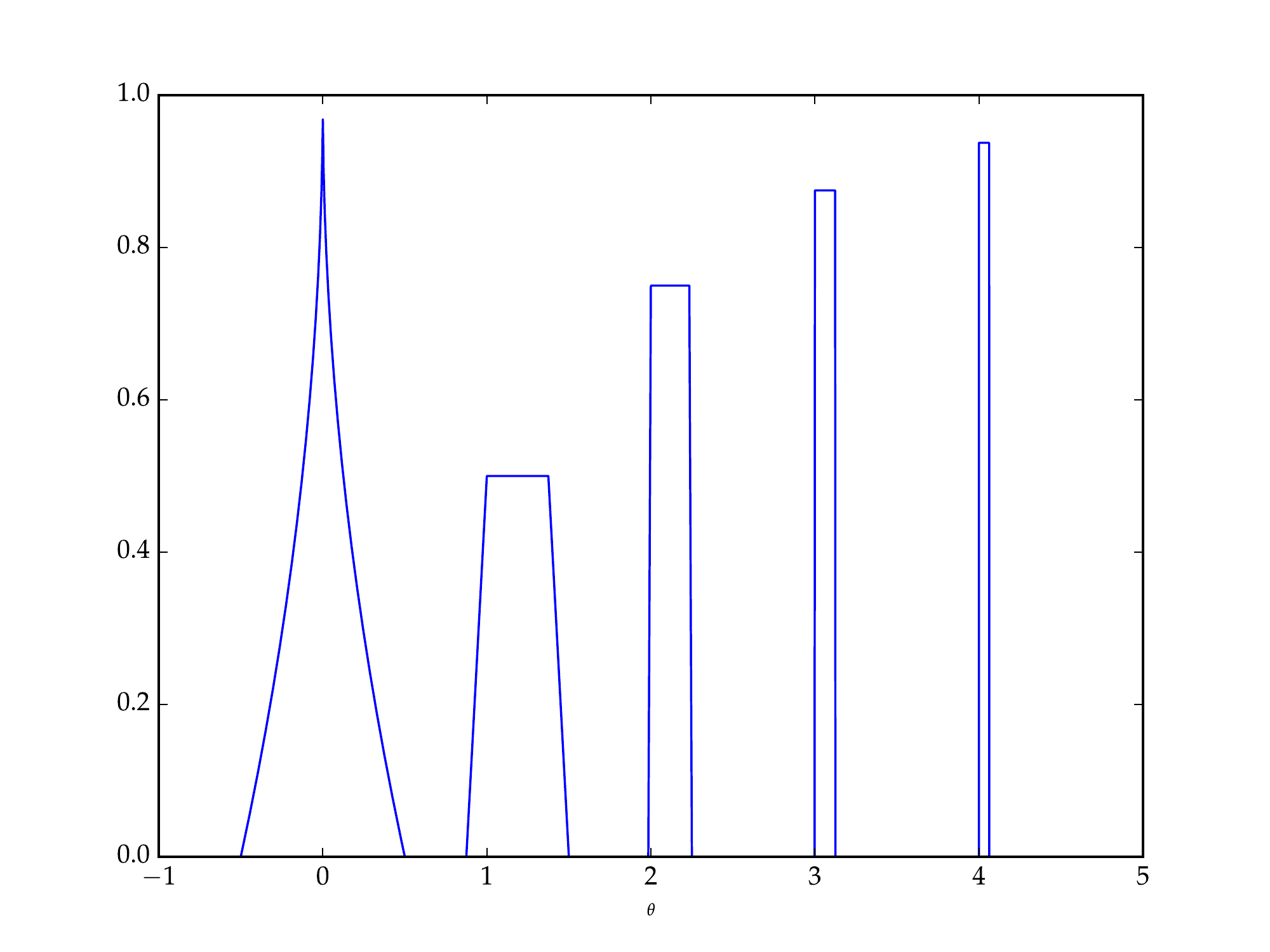}
  \caption{Posterior Density $\pi$}
  \label{fig:2}
\end{figure}

Note that the posterior density $\pi$ has a unique maximum, or 
\emph{posterior mode}, at $\theta=0$, where it takes the value $1$. Therefore,
the MAP estimator of this Bayesian model is $\hat{\theta}_{MAP}=0$.
On the other hand, considering the 0-1 
loss function $L^c$, we recall the equivalent definition of the associated 
Bayes estimator from equation \eqref{min2max}
\begin{equation}\label{eq2}\hat{\theta}^{c}_{B} \in \argmax_{\theta} \int_{\norm{z -\theta}<\frac{1}{c}} \pi(z)\,dz.\end{equation}
We next analyze the limiting behavior of the sequence of Bayes Estimators 
$\hat{\theta}^c_B$ when $c \to \infty$.

Consider the sequence of scalars $\left\{c^{\nu}=  2 \cdot 4^{\nu} :\nu \in \N_{+}\right\}$,
and the associated family of loss functions $L^{c^\nu}$. We will prove 
that for the Bayesian model that we have constructed,
 $\hat{\theta}^{c^{\nu}}_{B} \not \to \hat{\theta}_{MAP}$. This will
show that a limit of Bayes estimators with $0$-$1$ loss is not necessarily a MAP
estimator. In order to find the maximum in \eqref{eq2},
we can consider the maximum on each nonzero piece of $\pi$. For the
non-zero portion around the origin, the integral in \eqref{eq2} is
obviously maximized at $0$. Furthermore, for each ``bump'' of $\pi$,
the integral is maximized (perhaps non-uniquely) at the midpoint of
the interval where the density is constant. In order to show that
$\hat{\theta}_{B}^{c^{\nu}} \not \to \hat{\theta}_{MAP}$, it suffices to
show that for each $\nu$ there is a $\theta \not \in (-1/2,1/2)$ 
such that the evaluation of \eqref{eq2} at $\theta$ is strictly greater
than the evaluation at zero. This gives that $\hat{\theta}_{B}^{c^{\nu}}$
cannot have a limit point in $(-1/2, 1/2)$, and hence cannot have $0$,
the MAP estimator, as a limit point. We now perform the required
calculations.
\begin{enumerate}[i.]
\item Evaluation of \eqref{eq2} at $0$.
\begin{align}\label{Bayes0}
& \int_{\abs{z} < \frac{1}{2 \cdot 4^{\nu}}} 1- \sqrt{\abs{2z}}\,dz
= 2 \int_{0}^{\frac{1}{2 \cdot 4^{\nu}}} 1 - \sqrt{2z} \, dz 
= \frac{1}{4^{\nu}} - \frac{ 4 \sqrt{2}}{3 \cdot 2^{\frac{3}{2}}}
\cdot \frac{1}{8^{\nu}}
\end{align}

\item For each $\nu \in \N$, evaluating \eqref{eq2} at
 $\theta = 2 \nu + \frac{1}{2^{2\nu +1}}$ gives
\[\int_{\abs{z - \left( 2 \nu + \frac{1}{2^{2\nu +1}} \right)} 
< \frac{1}{2 \cdot 4^{\nu}}} \pi(z)\,dz 
=  \int_{[2^{\nu}, 2^{\nu} + \frac{1}{2^{2\nu}} - \frac{1}{8^{2\nu}}]}
\pi(z)\,dz + \int_{[2^{\nu} + \frac{1}{2^{\nu}}- \frac{1}{8^{2\nu}},
2^{\nu} + \frac{1}{2^{2\nu}}]} \pi(z)\,dz, \]
The first part of this sum is an integral over a constant piece of
$\pi$. The second
is a linear piece of $\pi$. Bounding the sum of the integrals below by 
the value of just the integral over the constant piece, we have that
\begin{align} \label{OtherRisk}
\int_{\abs{z - \left(2 \nu + \frac{1}{2^{2\nu +1}}\right)} < \frac{1}{2 \cdot 4^{\nu}}} \pi(z)\,dz \geq \left(1- \frac{1}{2^{2\nu}} \right) \, \left( \frac{1}{2^{2\nu}} -
\frac{1}{8^{2\nu}} \right) = \frac{1}{4^{\nu}} - \left(2 - \frac{1}{16^{\nu}}\right) \, \frac{1}{16^{\nu}}.
\end{align}
\end{enumerate}
Since \eqref{OtherRisk} is strictly greater than \eqref{eq2} for all
$\nu \geq 1$, we conclude that $\hat{\theta}^{c^{\nu}}_{B} \not \in
(\frac{-1}{2}, \frac{1}{2})$ for all $\nu \geq 1$. 
Hence the sequence
of Bayes estimators $\hat{\theta}^{c^{\nu}}_{B}$ has $c^{\nu} \to \infty$,
but does not have a MAP estimator as a limit point. This concludes the
counterexample and shows that the claim is false.

\section{Convergence results}\label{sec:5}
In this section we provide conditions on posterior distributions which
guarantee that a sequence of Bayes Estimators with $0$-$1$ loss has a
MAP estimator as a limit point. In addition, we provide a partial
converse by showing that each MAP estimator is a limit of approximate
Bayes estimators.

Define an estimator $\hat{\theta}$ of $\theta$ to be an
$\epsilon$-\emph{approximate Bayes estimator} with respect to loss $L$
if for each $x \in \mathcal{X}$.
\[\int_{z \in \Theta} L(\hat{\theta}(x), z) \pi(z|x) \, dz \geq
\sup_{\theta \in \Theta} \int_{z \in \Theta} L(\theta, z)
\pi(z|x) \, dz - \epsilon(x).\]
Here, $\epsilon$ is a function from $\mathcal{X}$ to $\R$.  We say that
$\hat{\theta}$ is a \emph{limit of approximate Bayes estimators} 
if there are sequences $V^{\nu}$, $\hat{\theta}^{\nu}$ and $\epsilon^{\nu}$ such
that $\hat{\theta}^{\nu}$ converges almost surely to $\hat{\theta}$,
$\hat{\theta}^{\nu}$ is an $\epsilon^\nu$-approximate Bayes estimator with
respect to loss $V^{\nu}$ for each $\nu$, and $\epsilon^\nu$ converges
$\mathcal{X}$-surely to $0$.

As discussed in the introduction, we assume that all densities are
upper semi-continuous and satisfy assumption \ref{assump}.

We begin with a deterministic result. The proof of the next lemma,
which is related to
the epi-convergence of mollifying approximates from \cite{VaAn}, is
included in the Appendix \ref{sec:appI}.

\begin{lemma} \label{hypoconv} 
Assume $f: \Theta \to \overline{\mathbb{R}}$
is an upper semi-continuous density. Let $s_{n}$ denote the volume of
the unit ball in $\mathbb{R}^{n}$. Define
 $$f^{\nu}(\theta) := \nu^{n} \cdot s_{n} \cdot \int_{\norm{\theta - z} < \frac{1}{\nu}} f(z) dz.$$
Then $f^{\nu}$ hypo-converges to $f$.
\end{lemma}

We note that even though upper semi-continuous densities may
take $\infty$ as a value, and hence need not be proper, $f^{\nu}$ must
be proper because $f$ integrates to one.

This lemma allows us to prove the following result.
\begin{theorem} \label{approxBayes}
If $\pi(\theta|x)$ is proper $\mathcal{X}$ almost surely, then any MAP 
estimator is a limit of approximate Bayes estimators.
\end{theorem} 
\begin{proof}
Let $\hat{\theta}_{MAP}$ be a maximum a posteriori estimator of
$\theta$. By definition, $\hat{\theta}_{MAP}(x) \in \argmax
\pi(\theta|x)$. Let $x$ be any element in
$\mathcal{X}$. For ease of notation, we will drop the explicit
dependence on $x$ by writing $\hat{\theta}_{MAP}: =
\hat{\theta}_{MAP}(x)$ and letting $f(\theta):=\pi(\theta|x)$. By lemma
\ref{hypoconv}, $f^{\nu} \xrightarrow{hypo} f$. From the definition of
hypo-convergence, there is a sequence $\theta^{\nu}(x) \to
\hat{\theta}_{MAP}(x)$ such that $\limsup_{\nu} f^{\nu}(\theta^{\nu})
\leq f(\hat{\theta}_{MAP})$. Also directly from the definition,
$\liminf_{\nu} f^{\nu}(\theta^{\nu}) \geq f(\hat{\theta}_{MAP})$. Hence
$\lim_{\nu} f^{\nu}(\theta^{\nu}) = f(\hat{\theta}_{MAP})$. 

Repeating this construction pointwise for each $x \in \mathcal{X}$, we define a
sequence of estimators with $\hat{\theta}^{\nu}(x)$ the
$\theta^{\nu}$ sequence above. Define $\epsilon^{\nu}(x)$ as
$\sup_{\theta \in \Theta} f^{\nu}(\theta) - f^{\nu}(\theta^{\nu})$. 
We claim that $L^{\nu}$,
$\hat{\theta}^{\nu}$, and $\epsilon^{\nu}$ satisfy the conditions so
that $\hat{\theta}_{MAP}$ is an approximate Bayes estimator. We must
verify the three conditions in the definition. The first two, that
$\hat{\theta}^{\nu}$ converges almost surely to $\hat{\theta}_{MAP}$
and that $\hat{\theta}^{\nu}$ is an $\epsilon^{\nu}$-approximate Bayes
estimator, are true by construction. Lastly, we must show that $\epsilon^{\nu}
\to 0$ almost surely. By monotonicity of the integral, we know that
$$f^{\nu}(\theta^{\nu}) \leq \sup_{\theta} f^{\nu}(\theta) \leq
f(\hat{\theta}_{MAP}).$$
For each $x\in \mathcal{X}$ with $\pi(\theta|x)$ proper, we combine
this inequality with the fact that
$\lim_{\nu} f^{\nu}(\theta^{\nu}) = f(\hat{\theta}_{MAP})$ to arrive
at the following
$$0 \leq \epsilon^{\nu}(x) = \sup f^{\nu} - f^{\nu}(\theta^{\nu}) \leq f(\hat{\theta}_{MAP}) -
f^{\nu}(\theta^{\nu}) \to 0$$
Note that $\pi(\theta|x)$ must be proper for this to hold: it
guarantees that $f(\hat{\theta}_{MAP}) < \infty$, and hence
$f(\hat{\theta}_{MAP}) = \lim_{\nu} f^{\nu}(\theta^{\nu})$ gives
$f(\hat{\theta}_{MAP}) - \lim_{\nu} f^{\nu}(\theta^{\nu}) =0$.

We conclude that $\epsilon^{\nu}(x) \to 0$ almost everywhere, since
$\pi(\theta|x)$ is proper $\mathcal{X}$ almost everywhere. This shows
the third condition in the definition, so that $\hat{\theta}_{MAP}$ is
a limit of approximate Bayes estimators.\qed
\end{proof}

The notion of an approximate Bayes
estimator captures the idea that Bayes and MAP estimators are close
in terms of evaluation of the objective. They need not be close in
distance to each other in $\Theta$. This point is subtle, and
underlies the confusion in the incorrect claim in \eqref{limitEst}.
In fact, as the counterexample 
in the previous section shows, Theorem \ref{approxBayes}
cannot be strengthened from approximate to true Bayes estimators
without additional conditions.

We turn now to providing those conditions, in a sort of converse to theorem
\ref{approxBayes}. We turn our focus to characterizing when a sequence of Bayes
estimators converges to a MAP estimator. We again begin with a deterministic
result.

\begin{theorem} \label{dettheorem}
Assume there exists $\alpha$ such that $\lev_{\geq \alpha}
f$ is bounded and has nonempty interior.  Then 
 $\argmax f^{\nu}$ is eventually nonempty, and
any sequence maximizers of $f^{\nu}$ as $\nu \to \infty$ has a 
maximizer of $f$ as a limit point.
\end{theorem}

\begin{proof}
Let $f$ be an usc density. Assume there is an $\alpha \in \mathbb{R}$
such that the upper level set
$\lev_{\geq \alpha} f$ is bounded and has nonempty interior.  Note
that $\lev_{\geq \alpha}$ is
compact by the Heine-Borel Theorem because upper level sets are closed
when $f$ is upper semicontinuous. 

First, we show that there is an
index $\nu_{0}$ such that $\argmax f^{\nu}$ is nonempty 
 for each $\nu > \nu_{0}$. Then we show that
any sequence of maximizers of $f^{\nu}$, has as a
limit point a maximizer of $f$.

To show the existence of a maximizer, we first show that $f^{\nu}$ is
upper semi-continuous with a bounded, nonempty upper level set. That
$f^{\nu}$ attains its maximum then follows because an upper
semi-continuous function attains its maximum on a compact set. We
recall from section \ref{sec:3} that our framework allows that
this maximum could be $\infty$.

Fix $\nu \in \mathbb{N}$ and  $\theta \in \mathbb{R}^{n}$ Let
$\{ \theta^{\eta} \}_{\eta \in \mathbb{N}}$ be any sequence that converges to
$\theta$.
To show upper semicontinuity we must have 
$$\limsup_{\eta} f^{\nu}(\theta^{\eta}) \leq f^{\nu}(\theta),$$
but this follows from the following chain of inequalities
\begin{align*}
\limsup_{\eta} f^{\nu}(\theta^{\eta})
&= \limsup_{\eta} s_{n} \cdot {\nu}^{n} \cdot \int_{\norm{\theta^{\eta}-z} < \frac{1}{\nu}} f(z) \, dz \\
&= \limsup_{\eta} s_{n} \cdot \nu^{n} \cdot \int_{\norm{z} < \frac{1}{\nu}} f(z+\theta^{\eta}) \, dz \\
&\leq s_{n} \cdot \nu^{n} \cdot \int_{\norm{z} < \frac{1}{\nu}} \limsup_{\eta} f(z+\theta^{\eta}) \, dz\\
&\leq s_{n} \cdot \nu^{n} \cdot \int_{\norm{z} < \frac{1}{\nu}} f(z+\theta) \, dz \\
&= f^{\nu}(\theta).
\end{align*}
The first inequality follows from Fatou's lemma, and the second from
the upper semicontinuity of $f$. Hence $f^{\nu}$ is upper
semicontinuous.

The family of functions $f^{\nu}$ has a bounded upper level set because $\lev_{\geq \alpha} f$
bounded with constant $M$ implies that $\lev_{\geq \alpha} f^{\nu}$ is
bounded with constant $M+\frac{1}{\nu}$. Lastly, we show that the
upper level set $\lev_{\geq \alpha} f^{\nu}$ is nonempty. Let $\theta \in
\setint \lev_{\geq \alpha} f^{\nu}$, and denote by $\delta$ radius
such that $\mathbb{B}(\theta, \delta) \subseteq \setint \lev_{\geq
\alpha} f^{\nu}$. Choose $\nu_{0} \geq
\frac{1}{\delta}$. Then for any $\nu > \nu_{0}$
$$f^{\nu}(\theta) = s_{n} \cdot \nu^{n} \int_{\norm{\theta-z}< \frac{1}{\nu}} f(z) \, dz \geq
s_{n} \cdot \nu^{n} \cdot \alpha \cdot \frac{1}{\nu^n} \cdot s_{n} = \alpha.$$
So $\theta \in \lev_{\geq \alpha} f^{\nu}$. We conclude that this level set
is nonempty and bounded, so that $f^{\nu}$ attains its maximum.

Now let $\hat{\theta}^{\nu}$ be any sequence of maximizers of
$f^{\nu}$. For $\nu > \nu_{0}$, 
$\hat{\theta}^{\nu} \in \lev_{\geq \alpha} f^{\nu}$.
From Lemma \ref{hypoconv} and Proposition \ref{HitMiss}
$$\liminf_{\nu} f^{\nu}(\hat{\theta}^{\nu}_{B}) = \liminf_{\nu}
\sup_{\mathbb{R}^{n}} f^{\nu} \geq \sup_{\mathbb{R}^{n}} f^{\nu}$$
and
\begin{align*}
& \limsup_{\nu} f^{\nu}(\hat{\theta}^{\nu}) \\
&= \limsup_{\nu}(\sup_{\mathbb{R}^{n}} f^{\nu}) \\
&=\limsup_{\nu} \left( \sup_{\overline{\mathbb{B}}(0,M+1)} f^{\nu}
\right) \\ 
&\leq \sup_{\overline{\mathbb{B}}(0, M+1)} f \\ 
&= \sup_{\mathbb{R}^{n}} f.
\end{align*}

So $\lim f^{\nu}(\hat{\theta}^{\nu}_{B}) = \sup_{\mathbb{R}^{n}} f.$

Since $\hat{\theta}^{\nu}$ is eventually in $\overline{\mathbb{B}}(0,
M+1)$, which is compact, it has a convergent subsequence $\hat{\theta}^{\nu_{k}} \to
\tilde{\theta}$. By upper semi-continuity
 $$f(\tilde{\theta}) \geq \limsup f(\hat{\theta}^{\nu_{k}}) = \lim
f(\hat{\theta}^{\nu}).$$
Hence $\tilde{\theta}$ maximizes $f$. Furthermore, by upper
semi-continuity any limit point of $\hat{\theta}^{\nu}$ maximizes $f$.
This proves the theorem.\qed
\end{proof}

We now prove our main result about the relationship between estimators.

\begin{theorem} \label{maintheorem}
Suppose that for each $x \in \mathcal{X}$, 
$\pi(\theta |x)$ satisfies the following property:
There exists $\alpha$ such that 
$$\left \{ \theta : \pi(\theta |x) > \alpha  \right \}$$
is bounded and has nonempty interior. Then for any of sequence of Bayes
estimators $\hat{\theta}_{B}^{\nu}$ with $0$-$1$ loss $L^{\nu}$
\begin{enumerate}[(i)]
\item There is a MAP estimator $\hat{\theta}_{MAP}$ such that
$\hat{\theta}_{MAP}(x)$ is a limit
point of $\hat{\theta}_{B}^{\nu}(x)$ for each $x$.
\item Every limit point of $\hat{\theta}_{B}^{\nu}(x)$ defines a
MAP estimator, in the sense that there is a MAP estimator
$\hat{\theta}_{MAP}$ with
$\hat{\theta}_{MAP}(x)$ equal to that limit point.
\end{enumerate}
\end{theorem}

\begin{proof}
Fix $x \in \mathcal{X}$. By assumption, 
$\pi(\theta|x)$ has a bounded level set with nonempty interior.
From the deterministic result, Theorem \ref{dettheorem}, the sequence
$\hat{\theta}_{B}^{\nu}(x)$ has a limit point $\tilde{\theta}(x)$,
and this limit point maximizes $\pi(\theta|x)$. Define a MAP estimator
$\hat{\theta}_{MAP}(x) := \tilde{\theta}(x)$ pointwise for each $x \in
\mathcal{X}$. This proves (i). 

For (ii), the result follows
since the method of defining a MAP estimator in the previous paragraph
is valid for every limit point of $\hat{\theta}^{\nu}_{B}(x)$.\qed
\end{proof}

As a consequence of the proof, we note the following: if
$\hat{\theta}_{B}^{\nu}$ converges to some estimator $\hat{\theta}$
almost surely, then $\hat{\theta}$ is a MAP estimator.


We next establish related results for various shape constrained
densities. Recall that a function $f: \mathbb{R}^{n} \to \overline{\mathbb{R}}$ is
quasiconcave if for each $x, y \in \mathbb{R}^{n}$ and $\lambda \in
[0,1]$ 
$$f(\lambda x + (1-\lambda) y) \geq \min \{ f(x), f(y) \}.$$
For a quasiconcave function, local maximality implies global
maximality. This shape constraint captures the notion of a
probability distribution having a ``single peak''. It is also
sometimes referred to as unimodality, but we avoid this terminology
because it has multiple meanings in the literature. 

\begin{theorem}\label{quasiconcave}
Assume that for each $x \in \mathcal{X}$, $\pi(\theta|x)$ is quasiconcave.
Then the conclusion of Theorem \ref{maintheorem} holds.
\end{theorem}

\begin{proof}
For each $x$, since $\int \pi(z|x) \, dz = 1$, there is an $\alpha >0$
such that $\mu(\{\theta:\pi(\theta|x) \geq \alpha \})$ has positive measure. By
quasiconvexity, $\lev_{\geq \alpha} \pi(\theta|x) = \{\theta |
\pi(\theta|x) \geq \alpha \}$ is convex. 

Note $\lev_{\geq \alpha}$ having an interior point is equivalent to
the set containing $n$ affinely independent points.
Suppose, towards a contradiction, that  $\lev_{\geq \alpha} \pi(\theta|x)$ does
not contain any interior points, then its affine hull lies in an $n-1$ dimensional
 plane. This contradicts the set having positive measure. Hence
$\{\theta : \pi(\theta|x)\}$ must have an interior point.

In order to apply Theorem \ref{maintheorem} we must also show that the level
set $\lev_{\geq \alpha} \pi(\theta|x)$ is bounded. Fix $B$ to be any n-dimensional ball in $\lev_{\geq
\alpha} \pi(\theta|x)$. If $\theta^{\nu}$ were a sequence in $\lev_{\geq
\alpha} \pi(\theta|x)$ such that $\norm{\theta^{\nu}} \to \infty$,
then
$$\int_{z \in \Theta} \pi(z |x) \, dz \geq \int_{z \in \lev_{\geq
\alpha} \pi(\theta |x)} \pi(z | x) \, dz \geq \mu(\{\conv(B
\cup \{\theta^{\nu} \})\}) \cdot \alpha$$
Here, $\mu$ denotes Lebesgue measure. One can easily show that $\mu(\conv(B \cup \{\theta^{\nu} \})) \to
\infty$ when $\theta^{\nu} \to \infty$. This contradicts that $\int
\pi(z|x) dz = 1$. Hence there cannot exist such
a sequence, so the level set is bounded and the result is proven.\qed
\end{proof}

The following corollary about log-concave densities follows
immediately. Recall
that a function $f: \mathbb{R}^n \to \mathbb{R}$ is log-concave if
for all $x,y \in \mathbb{R}^{n}$ and $\lambda \in [0,1]$
$$f(\lambda x + (1-\lambda)y) \geq f(x)^{\lambda}
f(y)^{1-\lambda}$$
Log-concave densities have appeared in recent work of
\cite{rufibach07}, \cite{dumbgen09} due to their attractive
computational and theoretical properties in nonparametric 
density estimation.  

\begin{corollary}
Assume that for each $x \in \mathcal{X}$, $\pi(\theta|x)$ is log-concave. 
Then the conclusion of theorem \ref{maintheorem} holds.
\end{corollary}
\begin{proof}
Log-concavity implies quasiconcavity, and the result follows
immediately from the previous theorem.\qed
\end{proof}


\bibliographystyle{plain}
\bibliography{BAMAP}   

\appendix
\section{Proof of Lemma 1}\label{sec:appI}

In this appendix we provide the proof of Lemma \ref{hypoconv}.

\begin{proof}
Let $\theta \in \mathbb{R}^{n}$. To show hypo-convergence, we must
show
that for each sequence $\theta^{\nu} \to \theta$, $\limsup_{\nu}
f^{\nu}(\theta^{\nu}) \leq f(x)$ and that there exists a sequence
$\theta^{\nu} \to \theta$ with $\liminf_{\nu} f^{\nu}(\theta^{\nu})
\geq f(\theta)$.

Fix $\epsilon >0$. Since $f$ is upper semi-continuous at $\theta$,
there
is a $\delta >0$ such that $\norm{z-\theta} < 2 \delta$ gives
$f(z) - f(\theta) < \epsilon$. 

Consider any sequence $\theta^{\nu} \to \theta$. We have that
$$f^{\nu}(\theta^{\nu})- f(\theta) = 
s_{n} \cdot \nu^{n} \cdot \int_{\norm{\theta^{\nu}-z}> \frac{1}{\nu}} (f(z) -
f(\theta)) dz = s_{n} \cdot \nu^{n} \cdot \int_{\norm{z} < \frac{1}{\nu}} 
(f(z+\theta^{\nu})-f(\theta)) \, dz.$$
Choose $\nu_{0} \in \mathbb{N}$ so that
$\norm{\theta-\theta^{\nu}}<\delta$ and
$\frac{1}{\nu} < \delta$ for each $\nu > \nu_{0}$. Then for any $\nu >
\nu_{0}$,
$$s_{n} \cdot \nu^n \cdot \int_{\norm{z} < \frac{1}{\nu}} \left(f(z+\theta^{\nu})-f(\theta) \right) \, dz
\leq s_{n} \cdot \nu^n \cdot \epsilon \cdot \int_{\norm{z} < \frac{1}{\nu}} \, dz = \epsilon.$$
Thus $\limsup_{\nu} f^{\nu}(\theta^{\nu}) \leq f(\theta)$.

To establish the second part of the hypo-convergence definition, we focus
our attention on constructing a sequence that satisfies the required
inequality.

Consider any $\eta \in \mathbb{N}$. Recall that $f^{\nu}$ is an upper
semi-continuous density. Let $C$ be the set where $f$ is continuous.
Because $C$ is dense, for each $\nu \in \mathbb{N}$, there is a
$y^{\nu} \in C$ such that $\norm{y^{\nu}-x} < \frac{1}{\nu}$.
Furthermore, $y^{\nu} \in C$ means that there is a
$\delta(y^{\nu},\eta) >0$ such that any $z \in \Theta$ which satisfies
 $\norm{y^{\nu} - z} < \delta(y^{\nu},\eta)$ also has 
$$\abs{f(y^{\nu}) - f(z)} < \frac{1}{\eta}.$$
Here we use function notation for $\delta$ to emphasize that $\delta$
depends on both $y^{\nu}$ and $\eta$.  

For each $\eta$, define a sequence such that 
$$z^{\nu, \eta} = \begin{cases} 
0 & \text{ when }  \frac{1}{\nu} >\delta(y^{1}, \eta) \\ 
y^{1} & \text{ when }  \delta(y^{2}, \eta) \leq \frac{1}{\nu} <
\delta(y^{1},\eta) \\
y^{2} & \text{ when }  \delta(y^{3}, \eta) \leq \frac{1}{\nu} <
\min\{\delta(y^{2}, \eta), \delta(y^{1},\eta) \} \\
y^{3} & \text{ when }  \delta(y^{4}, \eta) \leq \frac{1}{\nu} <
\min_{i \leq 4}{\delta(y^{i}, \eta)} \\
\vdots &  \vdots \end{cases}$$
Extracting a diagonal subsequence from the sequences generated
according to this procedure gives a sequence $\theta^{\nu}$ such that
$\theta^{\nu} \to \theta$ and $\frac{1}{\nu} < \delta(\theta^{\nu},
\nu)$. In
particular, $\abs{f(\theta^{\nu}) - f(z)} < \frac{1}{\nu}$ for $z$ with
$\norm{\theta^{\nu}-z} < \frac{1}{\nu} $.

Hence, for any $\epsilon > 0$, choosing $\nu > \frac{2}{\epsilon}$
gives
\begin{align*}
\abs{f^{\nu}(\theta^{\nu}) - f(\theta)} 
&\leq \abs{f^{\nu}(\theta^{\nu}) - f(\theta^{\nu})} +
\abs{f(\theta^{\nu}) -f(\theta)} \\
&\leq \frac{\epsilon}{2} + \frac{\epsilon}{2} = \epsilon
\end{align*}

We conclude that $\lim_{\nu} f^{\nu}(\theta^{\nu}) = f(\theta)$, so
the
result is proven.
\end{proof}


\end{document}